\documentclass[12pt, english]{article}
\usepackage[T2A]{fontenc}
\usepackage[cp1251]{inputenc}
\usepackage{babel}

\usepackage[centertags]{amsmath}
\usepackage{amsfonts}
\usepackage{amssymb}
\usepackage{amsthm}
\usepackage{newlfont}
\usepackage[active]{srcltx}
\usepackage{bm}

\newtheorem{theorem}{\bf Theorem}
\newtheorem{lemma}{\bf Lemma}

\begin{document}

\title{On upper and lower bounds for probabilities of combinations of events 
}

\author{Andrei N. Frolov \footnote{This research is supported by RFBR, project 18--01--00393}
\\ Dept. of Mathematics and Mechanics
\\ St.~Petersburg State University
\\ St. Petersburg, Russia
\\ E-mail address: Andrei.Frolov@pobox.spbu.ru}

\maketitle


{\abstract{ We derive new upper and lower bounds for probabilities 
that $r$ or at least $r$ from $n$ events occur. These bounds can 
turn to equalities. The method is discussed as well.
It works for measurable space and measures with sign, too. We also discuss variants of the
results for conditional probability of above events given $\sigma$-field. 
Taking expectations from both parts of inequalities for conditional probabilities
can yield better bounds for unconditional ones. 
} 
}

{\bf AMS 2010 subject classification:} 60E15, 60F15,

\medskip
{\bf Key words:}
Bonferroni inequalities, Chung--Erd\H{o}s inequality, 
bounds for probabilities of unions of events,
bounds for probabilities of combinations of events,
measure of unions, 
Borel--Cantelli lemma


\section{Introduction}

In papers [1--4], we derived upper and lower bounds for probabilities and
conditional probabilities (given a $\sigma$-algebra) that at least $r$ and exactly 
$r$ from $n$ events occur.
These bounds mainly involve a small numbers (usually two or three) of moments of
the sum $\xi_n$ of the indicators of events. Moreover, these moments were of small orders.
In [5], we improved our method of deriving for such bounds from [1--4] and we obtained new upper and lower bounds for probabilities and conditional probabilities of combinations of events mentioned above. 
In the last paper, we dealt with bounds involving a large numbers of moments of $\xi_n$
and the moments were of high orders. Such bounds are well known as Bonferroni inequalities. 
Note that using of $n$ moments of $\xi_n$ yields an equality (instead of bounds) which is called 
the Jordan formula. One can find generalizations of the Jordan formula and 
Bonferroni inequalities in [5]. In present paper, we derive new upper and lower bounds 
for probabilities and conditional probabilities of combinations of events when a small number
of moments of $\xi_n$ is used. We apply an improvement of the method from [5] and
use moments of small orders which calculation is simpler. 

Bounds for probabilities of combinations of events are of essential interest in
probability, statistics, combinatorics and their applications. In probability,
bounds for union of events are of special interest. Every such bound yields
a new variant of the Borel--Cantelli lemma. One can find various bounds
for combinations of events and their applications in [6--26] and references therein.
For Bonferroni inequalities, one can check references in [5] as well.

Let $(\Omega, \mathcal{F}, \mathbf{P})$ be a probability space and $\mathcal{A}$ be
a $\sigma$-field of events with $\mathcal{A}\subset\mathcal{F}$. For a collection of
events $A_1, A_2, \dots, A_n$, let $B_i$ denote the event that exactly $i$ events
occur from those $n$ events, where $i=0,1,\ldots,n$. Put
$$U_r= \bigcup\limits_{i=r}^n B_i$$
for $r=1,2,\ldots,n$. It is clear that $ U_r $ is the event that at least $r$ from
$n$ events $A_1, A_2, \dots, A_n$ occur. 
In the sequel, we deal with the probabilities
$$ p_i = \mathbf{P}(B_i) \quad\mbox{and} \quad P_r = P(U_r) = \sum\limits_{i=r}^n p_i 
$$
and the conditional probabilities
$$ p_i^\mathcal{A} = \mathbf{P}(B_i|\mathcal{A}) \quad\mbox{and}\quad
P_r^\mathcal{A} = P(U_r |\mathcal{A}) 
= \sum\limits_{i=r}^n p_i^\mathcal{A}, \quad 
$$
where $ i=0,1,\ldots,n$ and $ r = 1, 2, \ldots n. $

In this paper, we present new upper and lower bounds for these probabilities
by linear combinations of binomial type moments of the random variable
$$ \xi_n = \sum_{i=1}^n I_{A_i}, $$ where
$ I_{A_i} $ is the indicator of the event $A_i$.
Note that
$B_i=\{ \xi_n = i \}$ for all $i$ and $U_r = \{ \xi_n \geqslant r \}$ for all $r$.

Below, our bounds for conditional and unconditional probabilities are formally
very similar. One only need to replace $ \mathbf{P}(\cdot) $ by
$ \mathbf{P}(\cdot |\mathcal{A} ) $ in the definitions of involved quantities. 
Of course, one has to remember that numbers turn to
random variables and all inequalities hold almost surely (a.s.).
Nevertheless, every bound for the conditional probability yields a bound for the
unconditional one. Image, for example, that 
$$ p_i^\mathcal{A} \geqslant \eta \quad \mbox{a.s.}
$$
for some non-negative random variable $ \eta $. Then we get
$$ p_i = \mathbf{E} p_i^\mathcal{A} \geqslant \mathbf{E} \eta.
$$
In this way, one can get a sharper inequality than its analogue for the unconditional probability. 
In the example from [25], the events $A_i$ are related with the first component 
of a two-dimensional discrete random vector while the $\sigma$-field $ \mathcal{A} $ is generated 
by the second component of this vector. This is an example of natural settings when conditional probabilities given $ \sigma $-field appear. 
 

\section{Techniques and tools}

Our method combines two results from earlier papers. The first result allows to obtain upper and lower 
bounds for linear combination of non-negative numbers with non-negative weights. 
The second one contains decompositions of probabilities of combinations of events 
in sums of such linear combinations. Note that elements of decompositions can have no probabilistic senses. They only have to be non-negative. Finding bounds for every item of the sum in the decomposition 
by the first result, we easy arrive at desired bounds for probabilities.

We will use the following notations and agreements. All vectors from $ \mathbb{R}^k $ 
are columns and they are denoted by $ \mathbf{a} $, $ \mathbf{b} $, etc. while 
their coordinates are correspondingly denoted by $a_i$, $b_i$, etc. The vector $ \mathbf{0}_k $ is the origin of $ \mathbb{R}^k $ and the vector $ \mathbf{1}_k $ consists of $k$ ones. We write $ \mathbf{a} \leqslant \mathbf{b} $ when
$a_i \leqslant b_i$ for all $i=1,\ldots,k$. Relations $ \mathbf{a} \geqslant \mathbf{b} $,
$ \mathbf{a}< \mathbf{b} $ and  $ \mathbf{a} > \mathbf{b} $ are defined in the same way.
Symbol $^T$ denotes transposition, so $ \mathbf{a}^T \mathbf{b} $ is the scalar product
of $ \mathbf{a} $ and $ \mathbf{b} $.

We will apply the following result from [5]
on inequalities for linear combinations of non-negative numbers.

\begin{theorem}\label{th1}
Assume that $\mathbf{z}, \mathbf{v} \in  \mathbb{R}^n$ and $ \mathbf{z} \geqslant \mathbf{0}_n$.
Let $\mathbf{F}=\| f_{ki}\|_{k=1,i=1}^{\ell, n}$  be a $\ell\times n$ matrix with real entries, where
$2 \leqslant \ell \leqslant n$. Put $Z= \mathbf{z}^T \mathbf{v}$
and
\begin{eqnarray} \label{10}
\mathbf{s} = \mathbf{F} \mathbf{z}.
\end{eqnarray}
Assume that for some $ \mathbf{i} \in \mathbb{N}^\ell $  with
$1 \leqslant i_1< i_2< \cdots 
< i_\ell \leqslant n$, 
the vector
$\mathbf{a} \in \mathbb{R}^\ell$ is a solution of the following linear system:
\begin{eqnarray}\label{20}
 \mathbf{F}_{\mathbf{i}}^T \mathbf{a} = \mathbf{v}_{\mathbf{i}},
\end{eqnarray}
where
$\mathbf{F}_{\mathbf{i}} = \|f_{ki_q}\|_{k=1,q=1}^{\ell, \ell}$
and $ \mathbf{v}_{\mathbf{i}} = (v_{i_1}, v_{i_2},\dots, v_{i_{\ell}})^T $.
Suppose that $\mathbf{z}^\ast \in \mathbb{R}^n$ is a vector such that its subvector
$\mathbf{z}_{\mathbf{i}}^\ast = (z_{i_1}^\ast, z_{i_2}^\ast,\dots, z_{i_\ell}^\ast)^T$ 
satisfies to the system of linear equations
\begin{eqnarray} \label{30}
\mathbf{F}_{\mathbf{i}} \mathbf{z}_{\mathbf{i}}^\ast = \mathbf{s}
\end{eqnarray} 
and $z_{i}^\ast=0$ for all $i \neq i_q$, $1 \leqslant i \leqslant n$,
$1 \leqslant q \leqslant \ell$.

If $\mathbf{b}=\mathbf{F}^T \mathbf{a} \leqslant \mathbf{v}$, then
$Z \geqslant Z^\ast = (\mathbf{z}^\ast)^T \mathbf{v}
=  \mathbf{s}^T \mathbf{a}$.
If $\mathbf{F}^T \mathbf{a} \geqslant \mathbf{v} $, then 
$Z \leqslant Z^\ast$.
\end{theorem}

If $\mathbf{z}^T \mathbf{1}_n=1$,  then $\mathbf{z}$ is a distribution of probabilities
and $\mathbf{s}$ is a vectors of moments of this distribution. Hence, the matrix $\mathbf{F}$
and the vector $\mathbf{s}$ are correspondingly called the matrix of moments and the vector of moments
even when $\mathbf{z}$ forms a distribution after a normalizations.

Theorem \ref{th1} allows us to construct the vector $\mathbf{z}^\ast$ having the same 
moments as $\mathbf{z}$. It is clear that repeating the procedure with $\mathbf{z}=\mathbf{z}^\ast$,
we will again obtain $\mathbf{z}^\ast$. It follows that the inequality of Theorem \ref{th1}
can turn to equalities for some $\mathbf{z}$. Moreover, if $ \ell=n $ then we will
obtain an equality which presents $Z$ as a linear combinations of moments. 
Jordan's formula is an example of such the presentation in probability and combinatorics.

Mention on a selection of $\mathbf{i}$. Without loss of generality, we can assume that
all components of $\mathbf{v}$ are ones or zeros. Below, $\mathbf{F}$ will
be a matrix of binomial type moments. Then $b_u$ is a polynomial in $u$ 
with zeros and ones in some fixed points. Dealing with this polynomial yields 
small numbers of variants for $\mathbf{i}$ provided $\ell$ is small enough. 
We consider $\ell=2$ and $\ell=3$. For large $\ell$, one can use a computer. 


Turn to decompositions of probabilities of combinations of events.

Put 
$J_d =\{ j=(j_1,\ldots,j_d):\; j_k \in \mathbb{N}
\; \mbox{for all} \; 1 \leqslant k \leqslant d 
 \; \mbox{and}
\; 1 \leqslant j_1< j_2< \cdots < j_d \leqslant n \}$
for $d=1,2,\ldots, n$
and $J_0=\{0\}$. Since $j \in J_d$ can be a number, we use notation $j$ instead of $\mathbf{j}$
in this special case.

We need the next result from [1].

\begin{lemma}\label{L1} 
Let $d$ be a fixed integer number such that $ 0 \leqslant d \leqslant r $.
Put $p_{i,j} = \mathbf{P}(B_i A_{j_1} \ldots A_{j_d})$  and
$p_{i,j}^\mathcal{A} = \mathbf{P}(B_i A_{j_1} \ldots A_{j_d}|\mathcal{A})$ 
for all $j \in J_d$. (For $d=0$, we assume that  $ A_{j_1} \ldots A_{j_d} =\Omega $ 
and, consequently, $p_{i,j} = p_i$ and $p_{i,j}^\mathcal{A} = p_i^\mathcal{A}$.)

Then for every $r$ with $1\leqslant r \leqslant n$ the following relations hold:
\begin{eqnarray}
&& \label{40}
p_r = \sum\limits_{j \in J_d} \frac{p_{r,j}}{C_r^d} , \quad
P_r = \sum\limits_{j \in J_d} \sum\limits_{i=r}^n
\frac{ p_{i,j}}{C_i^d} ,
\end{eqnarray}
where $C_i^d= i!/(d! (i-d)!)$.

If $p_r$, $P_r$ and $ p_{i,j} $ are replaced in  (\ref{40})
by $  p_r^\mathcal{A} $, $  P_r^\mathcal{A} $ 
and  $ p_{i,j}^\mathcal{A} $ correspondingly, then 
the relations hold with probability $1$.
\end{lemma}

Further, we will apply Theorem 1 to bound every item with index $j$ in
decompositions $(\ref{40})$. Note that for $P_r$, items are sums over $i$.

\section{Inequalities for probabilities of combinations of events}

Assume that $1 \leqslant r \leqslant n$ and $0 \leqslant d \leqslant r $.
For every $j\in J_d$, put $\mathbf{z}(j)=(z_1(j),\ldots, z_{n-d+1}(j))^T$, where
$$z_{i}(j) = \frac{p_{i+d-1,j}}{C_{i+d-1}^d}\quad\mbox{for}\quad i=1,2,\ldots, n-d+1.$$
Note that $p_{0,j}=\cdots=p_{d-1,j}=0$ for all $j\in J_d$ and $d\geqslant 1$. 

Take $\ell$ with $\ell \leqslant n-d+1$.
Put
\begin{eqnarray}\label{f}
\mathbf{F} = \| C_{i+d-1}^{k+d-1}\|_{k=1,i=1}^{\ell,n-d+1},
\end{eqnarray}
where
$$ C_u^v = \frac{(u)_v}{v!}, \quad (u)_v = u (u-1)\times\cdots\times (u-v+1)
$$
for all natural $u$ and $v$. 
Since $C_u^v=0$ for $v>u$, we have $f_{ki}=0$ for $k>i$. It is also clear that
$f_{kk}=1$ for all $k$. 

Put $\mathbf{s}(j) = \mathbf{F} \mathbf{z}(j)$ for all $j\in J_d$.

By (\ref{40}), we have
$$ P_r = \sum\limits_{j \in J_d} Z(j) \quad\mbox{and}\quad p_r = \sum\limits_{j \in J_d} Z(j)
\quad\mbox{for}\quad Z(j) = (\mathbf{z}(j))^T \mathbf{v}
$$
provided $\mathbf{v}$ is chosen appropriately.
Applying Theorem \ref{th1} with $\mathbf{z}=\mathbf{z(j)}$,
we can estimate $Z(j)$ by
$Z^*(j) = (\mathbf{s}(j))^T \mathbf{a}$ for all $j\in J_d$.
By Theorem \ref{th1}, we also have
$ Z^*(j) = (\mathbf{z}^*(j))^T \mathbf{v}$
for all $j\in J_d$. 
Hence, bounds turn to equalities for $\mathbf{z}(j)=\mathbf{z}^8(j)$.
It follows that bounds for $Z(j)$ yield inequalities for probabilities of 
corresponding combinations of events and these inequalities are sharp.

By Lemma 2 from [5], we  have
\begin{eqnarray}\label{sk}
s_k(j) = \frac{d!}{(k+d-1)!}
\sum\limits_{u_1\neq \cdots\neq u_{k-1} \in \{1,\ldots,n\} \backslash \{j_1,\ldots,j_d\}}
 \mathbf{P}\left( A_{u_1}\ldots A_{u_{k-1}}  A_{j_1}\ldots A_{j_d} \right)
\end{eqnarray}
for $1\leqslant k\leqslant \ell$ and  $j\in J_d$. 
%
From the proof of this lemma, 
one can also see that
$$ s_k(j) = \frac{d!}{(k+d-1)!} \mathbf{E} (\xi_n-d)_{k-1} I_{A_{j_1} \ldots A_{j_d}} 
$$
for all $k$ and $j\in J_d$. 
Hence, $ s_k(j) $ are normalized binomial moments of $\xi_n$ centered at $d$ over
the intersection of events $ A_{j_1}, \ldots, A_{j_d} $. (For $d=0$, they
are binomial moments of $\xi_n$.)

Relation $ (\ref{sk}) $ shows that $ s_k(j) $ are normalized sums of probabilities  of intersections
of the events under consideration and our bounds are applicable in practice. 
This is a reason to use matrix $\mathbf{F}$ introduced above.

Note that $s_1(j) = \mathbf{P}\left(  A_{j_1}\ldots A_{j_d} \right)$
and $s_1(j) =1$ for $d=0$ in particular. Hence, we use moments of "zero" order
of $\xi_n$ as well. One can easily modify the proofs below for the case of moments of higher
orders, but one have to take into account that applications of higher moments yield
more rough inequalities. 

We derive new inequalities for $\ell=2$ and $\ell=3$ only. For large $\ell$,
one can obtain better bounds, but direct calculations will be complicated and 
computers can be used then. 

Our first result is as follows.

\begin{theorem}\label{th2}
Assume that $\ell=2$ and $\mathbf{s}(j) = \mathbf{F} \mathbf{z}(j)$ for all $j\in J_d$
with $\mathbf{F}$ from (\ref{f}). (In this case, relation (\ref{sk}) holds for
$s_1(j)$ and $s_2(j)$, $j\in J_d$.) 

If $r-d \geqslant 1$, then
\begin{eqnarray}\label{u1}
P_r \leqslant \sum_{j \in J_d}  \frac{s_2(j)}{C_{r}^{d+1}}
\quad\mbox{and}\quad
p_r \leqslant \sum_{j \in J_d}  \frac{s_2(j)}{C_{r}^{d+1}}.
\end{eqnarray}

If $n-r \geqslant 1$, then
\begin{eqnarray}\label{u2}
P_r \leqslant \sum_{j \in J_d}  
\frac{(C_n^{d+1}-C_r^{d+1}) s_1(j)-(C_n^d-C_r^d) s_2(j)}{C_n^{d+1} C_r^d - C_n^d C_{r}^{d+1}}
\;\;\mbox{and}\;\;
p_r \leqslant \sum_{j \in J_d} 
\frac{C_n^{d+1} s_1(j)-C_n^d s_2(j)}{C_n^{d+1} C_r^d - C_n^d C_{r}^{d+1}}.
\end{eqnarray}

If $r-d \geqslant 1$, then
\begin{eqnarray}\label{l1}
P_r \geqslant \sum_{j \in J_d}  
\frac{-C_{r-1}^{d+1} s_1(j)+C_{r-1}^d s_2(j)}{ C_n^{d+1} C_{r-1}^d - C_n^d C_{r-1}^{d+1}}
\quad\mbox{and}\quad
p_n \geqslant \sum_{j \in J_d}  
\frac{-C_{n-1}^{d+1} s_1(j)+C_{n-1}^d s_2(j)}{ C_n^{d+1} C_{n-1}^d - C_n^d C_{n-1}^{d+1}}.
\end{eqnarray}

If $r=d \geqslant 1$ and $n-r\leqslant 1$, then
\begin{eqnarray}\label{l2}
P_d \geqslant \sum_{j \in J_d}  
\frac{ (d+1) (m s_1(j)-d s_2(j))}{ (m+d) C_{m+d-1}^d}
\quad\mbox{and}\quad
p_d \geqslant \sum_{j \in J_d}  \left( s_1(j)-(d+1) s_2(j)\right),
\end{eqnarray}
where $m$ is a natural number with $1\leqslant m \leqslant n-d$.
Optimal value of $m$ is  
$$ m-1 \leqslant \frac{(d+1) s_2(j)}{ s_1(j)} \leqslant m
$$
for $  s_1(j)>0 $ and $1$ or $n-d$ otherwise.

Define $s_1^\mathcal{A}(j)$ and $s_2^\mathcal{A}(j)$ for $j\in J_d$  by 
the right-hand side of relation (\ref{sk}) with $\mathbf{P}(\cdot)$ replaced
by $\mathbf{P}(\cdot|\mathcal{A})$.  Inequalities (\ref{u1})--(\ref{l2}) hold a.s.
with $P_r^{\mathcal{A}}$ and $p_r^{\mathcal{A}}$ instead of $P_r$ and $p_r$
and $s_1^\mathcal{A}(j)$ and $s_2^\mathcal{A}(j)$ instead of 
$s_1(j)$ and $s_2(j)$,
correspondingly.

\end{theorem}

\begin{proof}
For $\mathbf{b} = \mathbf{F}^T \mathbf{a}$, we have
$$ b_u = C_{u+d-1}^{d} a_1 + C_{u+d-1}^{d+1} a_2
= C_{u+d-1}^{d} L(u) 
= \frac{u (u-1)\cdot\ldots\cdot (u-d+1)}{d!}  L(u) 
$$ 
for $ u=1,2,\ldots,n-d+1 $, where
$$ L(u) = a_1 + \frac{u-1}{d+1} a_2.
$$
Put
$$ b(u)= \frac{u (u-1)\cdot\ldots\cdot (u-d+1)}{d!} L(u) \quad\mbox{for}\quad u \in \mathbb{R}.
$$
Below we use properties of polynomial $b(u)$ to find variants for $\mathbf{i}$. 
Note that the degree of $b(u)$ is $d+1$. Hence it has no more than $d$ local extremes.
Since it has $d-1$ local extremes on negative half-line between zeros
in $0,-1,\ldots,-(d-1)$, there are no more than one
local extreme on positive half-line.

We consider two variants of $\mathbf{v}$ as follows:
$$  \left\lgroup \begin{matrix}
\mathbf{0}_{r-d} \cr \mathbf{1}_{n-r+1} 
\end{matrix} \right\rgroup
\quad\mbox{and}\quad
 \left\lgroup \begin{matrix}
\mathbf{0}_{r-d} \cr 1 \cr \mathbf{0}_{n-r} 
\end{matrix} \right\rgroup
$$
to get bounds for $P_r$ and $p_r$, correspondingly. Then vector $\mathbf{v}_{\mathbf{i}}$ 
can only be 
$$ a) \left\lgroup \begin{matrix}
0 \cr 0 
\end{matrix} \right\rgroup, \quad 
b) \left\lgroup \begin{matrix}
0 \cr 1 
\end{matrix} \right\rgroup, \quad
c) \left\lgroup \begin{matrix}
1 \cr 1 
\end{matrix} \right\rgroup
\quad\mbox{and}\quad
a) \left\lgroup \begin{matrix}
0 \cr 0 
\end{matrix} \right\rgroup, \quad 
b) \left\lgroup \begin{matrix}
0 \cr 1 
\end{matrix} \right\rgroup, \quad
c) \left\lgroup \begin{matrix}
1 \cr 0 
\end{matrix} \right\rgroup
$$
for $P_r$ and $p_r$, correspondingly. We will see below that $\mathbf{i}$ is
the same for $P_r$ and $p_r$ in cases a), b) and c), correspondingly.
Therefore, we deal with cases a)--c) for $P_r$ and $p_r$ simultaneously. 

We start with upper bounds. 

For $P_r$, coefficients $a_1$ and $a_2$ have to be such that
$b_u \geqslant 0$ for $u\leqslant r-d$  and $b_{u} \geqslant 1$ for $u \geqslant r-d+1$.
For $p_r$, we need $b_u \geqslant 0$ for $u\neq r-d+1$  and $b_{r-d+1} \geqslant 1$.
This follows from condition $\mathbf{b} \geqslant \mathbf{v}$ of Theorem \ref{th1}.

a) In this case, $a_1=a_2=0$ and, therefore, $b_u=0$ for all $u$.
So, we have no bounds. 

b) For $P_r$, we have $b_{i_1}=0$ and $b_{i_2}=1$ for some 
$i_1 \leqslant r-d < i_2 \leqslant n-d+1$. Hence, $b(u)$ has a local
extreme on $(0, i_1)$ and it is strictly increasing on $(i_1,i_3)$.
Then the local extreme is minimum and it is negative. It follows that
$i_1=1$ and $i_2=r-d+1$. This is the only option which gives    
$b_u \geqslant 0$ for $u\leqslant r-d$  and $b_{u} \geqslant 1$ for $u \geqslant r-d+1$.
For $p_r$, the choice of $\mathbf{i}$ is the same. Hence, bounds for
$P_r$ and $p_r$ coincide in this case.

Put $\mathbf{i}=(1,r-d+1)$. Then
$$ \mathbf{F}_\mathbf{i} = \left\lgroup \begin{matrix}
1 & C_{r}^d \cr  0 & C_{r}^{d+1} 
\end{matrix} 
\right\rgroup \quad\mbox{and}\quad
\mathbf{v}_\mathbf{i} = \left\lgroup \begin{matrix}
0 \cr 1 
\end{matrix} \right\rgroup.
$$
The solution of system (\ref{20}) is
$$ \mathbf{a} = \frac{1}{C_{r}^{d+1}}
\left\lgroup \begin{matrix}
0 \cr 1
\end{matrix} \right\rgroup
$$
and we get (\ref{u1}).

c) For $P_r$, we have $b_{i_1}=1$ and $b_{i_2}=1$ for some 
$r-d+1 \leqslant i_1 < i_2 \leqslant n-d+1$.
Hence, $b(u)$ has a local extreme on $(i_1,i_2)$ and it is strictly
increasing on $(0, i_1)$. Hence, $i_1=r-d+1$ and $i_2=n-d+1$
which is the only variant to satisfy
$b_u \geqslant 0$ for $u\leqslant r-d$  and $b_{u} \geqslant 1$ for $u \geqslant r-d+1$.

For $p_r$,  we have $b_{r-d+1}=1$ and $b_{i_2}=0$ for some 
$r-d+2 \leqslant i_2 \leqslant n-d+1$. It yields that
$b(u)$ has a local extremes on $(0,i_2)$. Then $i_2=n-d+1$
to satisfy $b_u \geqslant 0$ for $u\neq r-d+1$  and $b_{r-d+1} \geqslant 1$.

%
%

It follows that $\mathbf{i}$ is the same for $P_r$ and $p_r$ while $\mathbf{v}_{\mathbf{i}}$
is different in (\ref{20}).

Put $\mathbf{i}=(r-d+1,n-d+1)$. Then
$$ \mathbf{F}_\mathbf{i} = \left\lgroup \begin{matrix}
C_r^d & C_{n}^d \cr  C_r^{d+1} & C_{n}^{d+1} 
\end{matrix} 
\right\rgroup, 
\quad
\mathbf{v}_\mathbf{i} = \left\lgroup \begin{matrix}
1 \cr 1 
\end{matrix} \right\rgroup
\quad\mbox{or}\quad
\mathbf{v}_\mathbf{i} = \left\lgroup \begin{matrix}
1 \cr 0
\end{matrix} \right\rgroup
$$
for $P_r$ and $p_r$, correspondingly.

The solutions of system (\ref{20}) are
$$ \mathbf{a} = \frac{1}{C_n^{d+1} C_r^d - C_n^d C_{r}^{d+1}}
\left\lgroup \begin{matrix}
C_n^{d+1}-C_r^{d+1} \cr C_r^{d}-C_n^{d}
\end{matrix} \right\rgroup
\quad\mbox{and}\quad
\mathbf{a} = \frac{1}{C_n^{d+1} C_r^d - C_n^d C_{r}^{d+1}}
\left\lgroup \begin{matrix}
C_n^{d+1} \cr - C_{n}^{d}
\end{matrix} \right\rgroup
$$
for $P_r$ and $p_r$, correspondingly. 
Hence, we get (\ref{u2}).

Turn to lower bounds. 

Three options for $\mathbf{v}_\mathbf{i}$ are the same
as for the upper bounds, but in case of lower bounds, we have another restrictions
on $b_u$. For $P_r$, we need $b_u\leqslant 0$
for $u \leqslant r-d$ and $b_u\leqslant 1$ for $n \geqslant r-d+1$. 
For $p_r$, we need $b_u \leqslant 0$ for $u\neq r-d+1$  and $b_{r-d+1} \leqslant 1$.
This is condition $\mathbf{b} \leqslant \mathbf{v}$ of Theorem \ref{th1}.

We deal with cases a)--c) again

a) For this option, we get trivial bounds by zero. 

b) For $P_r$, we have $b_{i_1}=0$ and $b_{i_2}=1$ for some 
$i_1 \leqslant r-d < i_2 \leqslant n-d+1$.  
It yields that $b(u)$ has a local extreme on $(0, i_1)$ and it is
strictly increasing on $(i_1,i_2)$. Hence, $i_1=r-d$ and $i_2=n-d+1$
which can only give $b_u\leqslant 0$
for $u \leqslant r-d$ and $b_u\leqslant 1$ for $n \geqslant r-d+1$. 

For $p_r$, we have $b_{i_1}=0$ and $b_{r-d+1}=1$ for some 
$i_1 \leqslant r-d$.
It follows that $b(u)$ has a local extreme on $(0, i_1)$ and it is
strictly increasing on $(i_1,r-d+1)$. Hence, $i_1=r-d$, $i_2=n-d+1$ and $r=n$
is the only way to have 
$b_u \leqslant 0$ for $u\neq r-d+1$  and $b_{r-d+1} \leqslant 1$.
 
It turns out that system (\ref{20}) is the same for $P_r$ and $p_r$, but
we can get a bound for $p_r$ for $r=n$ only.


Put $\mathbf{i}=(r-d,n-d+1)$. Then
$$ \mathbf{F}_\mathbf{i} = \left\lgroup \begin{matrix}
C_{r-1}^d & C_{n}^d \cr  C_{r-1}^{d+1} & C_{n}^{d+1} 
\end{matrix} 
\right\rgroup \quad\mbox{and}\quad
\mathbf{v}_\mathbf{i} = \left\lgroup \begin{matrix}
0 \cr 1 
\end{matrix} \right\rgroup
$$

The solution of system (\ref{20}) is
$$ \mathbf{a} = \frac{1}{C_n^{d+1} C_{r-1}^d - C_n^d C_{r-1}^{d+1}}
\left\lgroup \begin{matrix}
-C_{r-1}^{d+1} \cr C_{r-1}^{d}
\end{matrix} \right\rgroup
$$
and we arrive at (\ref{l1}).

c)  For $P_r$, we have $b_{i_1}=1$ and $b_{i_2}=1$ for some 
$r-d+1 \leqslant i_1 < i_2 \leqslant n-d+1$.  
Then $b(u)$ has a local extreme on $(i_1, i_2)$ and it is
strictly increasing on $(0,i_1)$. Hence, the case $r=d$ and $i_2=i_1+1$ 
only gives $b_u\leqslant 0$
for $u \leqslant r-d$ and $b_u\leqslant 1$ for $n \geqslant r-d+1$. 

For $p_r$, we have $b_{r-d+1}=1$ and $b_{i_2}=0$ for some 
$r-d+2 \leqslant i_2 \leqslant n-d+1$. Then $b(u)$ has a local extreme on $(0, i_2)$.
Then $r=d$ and $i_2=r-d+2=2$ and this choice
can only yield $b_u\leqslant 0$
for $u \neq r-d+1$ and $b_{r-d+1}\leqslant 1$. 

Assume that $r=d>0$.
For $P_r$, put $\mathbf{i}=(m,m+1)$ for $r-d+1 \leqslant m \leqslant n-d$. Then 
$$ \mathbf{F}_\mathbf{i} = \left\lgroup \begin{matrix}
C_{m+d-1}^d & C_{m+d}^d \cr  C_{m+d-1}^{d+1} & C_{m+d}^{d+1} 
\end{matrix} 
\right\rgroup \quad\mbox{and}\quad
\mathbf{v}_\mathbf{i} = \left\lgroup \begin{matrix}
1 \cr 1 
\end{matrix} \right\rgroup
$$
For $p_r$, we take $m=1$ and  $ \mathbf{v}_\mathbf{i} = (1,0)^T $.

Note that
$$\mathbf{F}_\mathbf{i}^{-1} = \frac{1}{(m+d) C_{m+d-1}^d}
\left\lgroup \begin{matrix}
m(m+d) & -(d+1)(m+d) \cr  -m(m-1) & (d+1)m
\end{matrix} 
\right\rgroup.
$$ 

The solutions of system (\ref{20}) are
$$ \mathbf{a} = \frac{d+1}{(m+d) C_{m+d-1}^d}
\left\lgroup \begin{matrix}
m \cr -d 
\end{matrix} \right\rgroup
\quad\mbox{and}\quad
\mathbf{a} = 
\left\lgroup \begin{matrix}
1 \cr -(d+1) 
\end{matrix} \right\rgroup
$$
for $P_r$ and $p_r$, correspondingly. This implies (\ref{l2}).

Optimize $m$ in the bound for $P_r$. By (\ref{30}), 
$\mathbf{z}_{\mathbf{i}}^*(j)=\mathbf{F}_\mathbf{i}^{-1} \mathbf{s}$.
This yields that
\begin{eqnarray*}
z^*_m(j) = \frac{m s_1(j) -(d+1) s_2(j)}{ C_{m+d-1}^d},\quad
z^*_{m+1}(j)=\frac{m ( -(m-1) s_1(j) +(d+1) s_2(j))}{(m+d) C_{m+d-1}^d}.
\end{eqnarray*}
Inequalities $ z^*_{m}(j) \geqslant 0 $ and 
$ z^*_{m+1}(j) \geqslant 0 $ implies that
$$ m \geqslant \frac{(d+1) s_2(j)}{ s_1(j)}
\quad\mbox{and}\quad
m-1 \leqslant \frac{(d+1) s_2(j)}{ s_1(j)}
$$
provided $  s_1(j)>0 $.

We now turn to bounds for the conditional probability given $\sigma$-field $\mathcal{A}$.
Fix variants of all random variables in the equality 
defining $P_r^\mathcal{A}$. This equality holds for all $\omega\in \Omega\setminus \mathcal{N}$,
where $\mathbf{P}(\mathcal{N})=0$. For every fixed $\omega\in \Omega\setminus \mathcal{N}$
we prove bounds similar to (\ref{u1})--(\ref{l2}) in the same way as before. 
(Note that our method do not depend on probability background. It works for numbers as well.)
As a result, we get inequalities (\ref{u1})--(\ref{l2}) for the variants of random variables
chosen before. If we replace one of these random variables by another variant then
the inequalities may fail on some set 
of zero probability.
We deal with a finite number of random variables. Hence, such replacements
of one or several random variables may fail our inequalities only on a set
of zero probability. It follows that inequalities (\ref{u1})--(\ref{l2}) hold a.s.
\end{proof}

Turn to the case $\ell=3$. Start with the following result for upper bounds.

\begin{theorem}\label{th3}
Assume that $\ell=3$ and $\mathbf{s}(j) = \mathbf{F} \mathbf{z}(j)$ for all $j\in J_d$
with $\mathbf{F}$ from (\ref{f}). (In this case, relation (\ref{sk}) holds for
$s_1(j)$, $s_2(j)$ and $s_3(j)$, $j\in J_d$.) 

If $r-d \geqslant 2$ 
then
\begin{eqnarray}
\label{Pr3ub1}
P_r \leqslant  \sum_{j \in J_d} {\bm\alpha}^{T} \mathbf{s}(j)
\quad\mbox{and}\quad
p_r \leqslant  \sum_{j \in J_d} {\bm\alpha}^{T} \mathbf{s}(j)
\end{eqnarray}
where
\begin{eqnarray}
\label{al}
 {\bm\alpha} = \frac{1}{\Delta C_{r}^{d}} \left\lgroup \begin{matrix}
m(m-1) \cr - 2(d+1) (m-1) \cr (d+1)(d+2) \end{matrix} \right\rgroup ,\quad
\Delta= (r-d-m)(r-d-m+1)
\end{eqnarray}
and $m$ is an arbitrary natural number such that 
$1 \leqslant m \leqslant r-d-1$. The optimal value of $m$ is 
\begin{eqnarray}
\label{mop}
 m-1  \leqslant (d+1)
\frac{ (r-d-1) s_2(j)-(d+2) s_3(j)}{(r-d) s_1(j)- (d+1) s_2(j)  }\leqslant m
\end{eqnarray}
for 
$(r-d) s_1(j)-(d+1) s_2(j) > 0$ and $1$ or $r-d-1$
otherwise.

If $r-d \geqslant 1$ and $n-r \geqslant 1$, then
\begin{eqnarray}
\label{Pr3ub2}
P_r \leqslant  \sum_{j \in J_d} {\bm\beta}^{T} \mathbf{s}(j)\quad\mbox{and}\quad
p_r  \leqslant  \sum_{j \in J_d} {\bm\delta}^{T} \mathbf{s}(j),
\end{eqnarray}
where
$$
{\bm\beta} =  \frac{1}{\Delta_1}
\left\lgroup \begin{matrix}
0 \cr C_n^{d+2}-C_{r}^{d+2} \cr C_{r}^{d+1} - C_n^{d+1}
\end{matrix} \right\rgroup, \quad 
{\bm\delta} =  \frac{1}{\Delta_1}
\left\lgroup \begin{matrix}
0 \cr C_n^{d+2} \cr  - C_n^{d+1}
\end{matrix} \right\rgroup \quad \mbox{and}\quad
\Delta_1= C_n^{d+2} C_{r}^{d+1} - C_n^{d+1} C_{r}^{d+2}.
$$

If 
$n-r \geqslant 2$, then
\begin{eqnarray}
\label{Pr3ub3}
P_r \leqslant  \sum_{j \in J_d} {\bm\gamma}^{T} \mathbf{s}(j)\quad\mbox{and}\quad
p_r \leqslant  \sum_{j \in J_d} {\bm\alpha}^{T} \mathbf{s}(j),
\end{eqnarray}
where  
$ {\bm\gamma} = (\gamma_{1}, \gamma_{2}, \gamma_{3})^T $ with
\begin{eqnarray*}
&& \hspace*{-\parindent}
 \gamma_{1}= \frac{m(r-d)(r-d-m)}{C_{m+d-1}^d \Delta } - \frac{(m-1)(r-d)(r-d-m+1)}{C_{m+d}^d \Delta }
+\frac{m(m-1)}{C_{r}^d \Delta },\\
&&   \hspace*{-\parindent}
\gamma_{2}= (d+1)\left(
-\frac{(r-d-m)(r-d+m-1)}{C_{m+d-1}^d \Delta } + \frac{(r-d-m+1)(r-d+m-2)}{C_{m+d}^d \Delta }
-\frac{2(m-1)}{C_{r}^d \Delta }\right),\\
&&  \hspace*{-\parindent}
\gamma_{3}= (d+1)(d+2)\left(\frac{r-d-m}{C_{m+d-1}^d \Delta } - \frac{r-d-m+1}{C_{m+d}^d \Delta }
+\frac{1}{C_{r}^d \Delta }\right),
\end{eqnarray*}
$ \Delta $ and $ \bm\alpha $ are from (\ref{al})
and $m$ is an arbitrary natural number such that $r-d+2 \leqslant m \leqslant n-d$. 
The optimal value for $m$ is defined by (\ref{mop})
provided $ (r-d) s_1(j) -(d+1) s_2(j) < 0$ and it is $r-d+2$ or $n-d$ otherwise.

If  $r-d \geqslant 2$ and $n-r \geqslant 2$, then
\begin{eqnarray}
\label{Pr3ub}
P_r \leqslant 
\sum_{j \in J_d} \min\left\{ {\bm\alpha}^{T} \mathbf{s}(j),
{\bm\beta}^{T} \mathbf{s}(j), {\bm\gamma}^{T} \mathbf{s}(j)
\right\} \;\mbox{and}\;\;
p_r \leqslant 
\sum_{j \in J_d} \min\left\{ {\bm\alpha}^{T} \mathbf{s}(j),
{\bm\delta}^{T} \mathbf{s}(j) 
\right\}.
\end{eqnarray}

Define $s_1^\mathcal{A}(j)$, $s_2^\mathcal{A}(j)$ and $s_3^\mathcal{A}(j)$ for $j\in J_d$  by 
the right-hand side of relation (\ref{sk}) with $\mathbf{P}(\cdot)$ replaced
by $\mathbf{P}(\cdot|\mathcal{A})$. Inequalities (\ref{Pr3ub1}), (\ref{Pr3ub2})--(\ref{Pr3ub}) hold a.s.
provided one replaced $P_r$ and $p_r$ by  $P_r^{\mathcal{A}}$ and $p_r^{\mathcal{A}}$ and
$s_1(j)$, $s_2(j)$ and $s_3(j)$ by
$s_1^\mathcal{A}(j)$, $s_2^\mathcal{A}(j)$ and $s_3^\mathcal{A}(j)$
correspondingly. In this case, optimal $m$ are random variables.

\end{theorem}

\begin{proof}


For $\mathbf{b} = \mathbf{F}^T \mathbf{a}$, we have
$$ b_u = C_{u+d-1}^{d} a_1 + C_{u+d-1}^{d+1} a_2 + C_{u+d-1}^{d+2} a_3
= C_{u+d-1}^d Q(u) = \frac{u (u+1)\cdot \ldots \cdot (u-d+1)}{d!} Q(u)
$$ 
for $ u=1,2,\ldots,n-d+1,$ where
$$ Q(u) = a_1 + \frac{u-1}{d+1} a_2 + \frac{(u-1)(u-2)}{(d+1)(d+2)} a_3. 
$$
Put 
\begin{eqnarray}\label{bu}
b(u) =  \frac{u (u+1)\cdot \ldots \cdot (u-d+1)}{d!} Q(u),\quad u \in \mathbb{R}.
\end{eqnarray}
It is clear that $b(u)=b_u$ for natural $u$ and $b(u)$ is a polynomial with degree $d+2$ (or less)
and $d$ zeros at $0, -1, \ldots, -(d-1)$. It follows that $b(u)$ has $d-1$ local extremes
on negative half-line and it can have no more than two extremes on positive half-line. 
This will repeatedly be used to find variants for $\mathbf{i}$.

We consider the following two variants of vector $\mathbf{v}$:
\begin{eqnarray}\label{vv}
\left\lgroup \begin{matrix}
\mathbf{0}_{r-d} \cr \mathbf{1}_{n-r+1} 
\end{matrix} \right\rgroup
\quad\mbox{and}\quad
\left\lgroup \begin{matrix}
\mathbf{0}_{r-d} \cr 1 \cr \mathbf{0}_{n-r} 
\end{matrix} \right\rgroup
\end{eqnarray}
to obtain bounds for $P_r$ and $p_r$, correspondingly. 
Then vector $\mathbf{v}_\mathbf{i}$ can only be as follows:
\begin{eqnarray}\label{ad}
 a) \left\lgroup \begin{matrix}
0 \cr 0 \cr 0
\end{matrix} \right\rgroup\!, \;\; 
 b) \left\lgroup \begin{matrix}
0 \cr 0 \cr 1
\end{matrix} \right\rgroup\!, \;\; 
c) \left\lgroup \begin{matrix}
0 \cr 1 \cr 1
\end{matrix} \right\rgroup\!, \;\;
d) \left\lgroup \begin{matrix}
1 \cr 1 \cr 1
\end{matrix} \right\rgroup
\;\; \mbox{and} \;\;
a) \left\lgroup \begin{matrix}
0 \cr 0 \cr 0
\end{matrix} \right\rgroup\!, \;\; 
b)  \left\lgroup \begin{matrix}
0 \cr 0 \cr 1
\end{matrix} \right\rgroup\!, \;\;
c) \left\lgroup \begin{matrix}
0 \cr 1 \cr 0
\end{matrix} \right\rgroup\!, \;\;
d) \left\lgroup \begin{matrix}
1 \cr 0 \cr 0
\end{matrix} \right\rgroup
\end{eqnarray}
for $P_r$ and $p_r$, correspondingly. We will show below that 
$\mathbf{i}$ is the same for $P_r$ and $p_r$ in cases a), b), c) and d),
correspondingly. Therefore, we deal with each case from a)--d) for $P_r$ and $p_r$ simultaneously.

In case of $P_r$, coefficients $a_1$, $a_2$ and $a_3$ have to be such that
$b_u \geqslant 0$ for  $u \leqslant r-d$ and $b_u \geqslant 1$ for $u \geqslant r-d+1$.
In case of $p_r$, we need  $b_u \geqslant 0$ for  $u \neq r-d+1$ and
$b_{r-d+1} \geqslant 1$.
This follows from condition $\mathbf{b}\geqslant \mathbf{v}$ of Theorem \ref{th1}.

a) In this case, $a_1=a_2=a_3=0$ and, therefore, $b_u=0$ for all $u$.
Hence, we have no bounds for $P_r$ and $p_r$ both. 

b) For $P_r$, we have
$b_{i_1}=b_{i_2}=0$ and $b_{i_3}=1$ for some $i_1<i_2 \leqslant r-d < i_3 \leqslant n-d+1$.
Then there are two extremes of $b(u)$ on intervals $(0,i_1)$ and $(i_1,i_2)$. Hence, $b(u)$ is
strictly increasing on $(i_2, i_3)$. This yields that $i_3=r-d+1$. Moreover,
$b(u)$ has a local minimum on $(i_1,i_2)$ and this minimum is negative. It follows
that $i_2=i_1+1$. This choice of $i_1, i_2, i_3$ only implies
that $b_u \geqslant 0$ for  $u \leqslant r-d$ and
$b_u \geqslant 1$ for $u \geqslant r-d+1$. For $p_r$, we have the same option of $\mathbf{i}$.
Since $\mathbf{v}_{\mathbf{i}}$ is the same for $P_r$ and $p_r$ both, 
bounds for $P_r$ and $p_r$ coincide.

Put $\mathbf{i}=(m,m+1,r-d+1)$, where $1 \leqslant m \leqslant r-d-1$. 
Here $m$ is a parameter which can be chosen to optimize bounds below.

Then
$$ \mathbf{F}_\mathbf{i} = \left\lgroup \begin{matrix}
C_{m+d-1}^{d} & C_{m+d}^{d} & C_r^{d} \cr
C_{m+d-1}^{d+1} & C_{m+d}^{d+1} & C_r^{d+1} \cr
C_{m+d-1}^{d+2} & C_{m+d}^{d+2} & C_r^{d+2} 
\end{matrix} 
\right\rgroup \quad\mbox{and}\quad
\mathbf{v}_\mathbf{i} = \left\lgroup \begin{matrix}
0 \cr 0 \cr 1
\end{matrix} \right\rgroup.
$$

Note that 
$$ \mathbf{F}^{-1}_\mathbf{i} = \left\lgroup \begin{matrix}
\frac{m(r-d)(r-d-m)}{C_{m+d-1}^d \Delta }& 
-\frac{(r-d-m)(r-d+m-1)(d+1)}{C_{m+d-1}^d \Delta }& 
\frac{(r-d-m)(d+1)(d+2)}{C_{m+d-1}^d \Delta } \cr
 - \frac{(m-1)(r-d)(r-d-m+1)}{C_{m+d}^d \Delta }& 
 \frac{(r-d-m+1)(r-d+m-2)(d+1)}{C_{m+d}^d \Delta }& 
- \frac{(r-d-m+1)(d+1)(d+2)}{C_{m+d}^d \Delta } \cr
\frac{m(m-1)}{C_{r}^d \Delta} & 
-\frac{2(m-1)(d+1)}{C_{r}^d \Delta } &
\frac{(d+1)(d+2)}{C_{r}^d \Delta }  
\end{matrix} 
\right\rgroup. 
$$
Then the solution of system (\ref{20}) is $ \mathbf{a}=\bm{\alpha} $
and we get (\ref{Pr3ub1}).

Turn to an optimization over $m$.  By (\ref{30}), we have
$\mathbf{z}^*_{\mathbf{i}}(j) = \mathbf{F}^{-1}_\mathbf{i} \mathbf{s}(j)$.
Hence, we get
\begin{eqnarray*}
&& \hspace*{-\parindent}
z^*_{m}(j) = \frac{r-d-m}{C_{m+d-1}^d \Delta } \left(
m(r-d) s_1(j) -(r-d+m-1)(d+1) s_2(j) + (d+1)(d+2) s_3(j) \right),
\\ && \hspace*{-\parindent}
z^*_{m+1}(j) = \frac{r-d-m+1}{C_{m+d}^d \Delta } \left(
 - (m-1)(r-d) s_1(j)+  \right.
\\ && \hspace*{9\parindent} \left. +
 (r-d+m-2)(d+1) s_2(j) - (d+1)(d+2) s_3(j) \right), 
\\ && \hspace*{-\parindent}
z^*_{r-d+1}(j) = \frac{1}{C_{r}^d \Delta} \left(
m(m-1) s_1(j) -2(m-1)(d+1) s_2(j) + (d+1)(d+2) s_3(j) \right).
\end{eqnarray*}
The inequalities $ z^*_{m}(j)\geqslant 0 $ and $ z^*_{m+1}(j)\geqslant 0 $ give (\ref{mop})
provided 
$ (r-d) s_1(j)- (d+1) s_2(j) > 0$.



c) For $P_r$, we have $b_{i_1}=0$ and $b_{i_2}=b_{i_3}=1$ for some $1\leqslant i_1 \leqslant r-d$ 
and $r-d+1 \leqslant i_2 < i_3 \leqslant n-d+1$.
Then $b(u)$ has local extremes on $(0,i_1)$ and $(i_2, i_3)$.
Hence, $b(u)$ is strictly increasing on $(i_1, i_2)$. It follows that
$i_2=r-d+1$. Moreover, $b(u)$ has a local minimum on 
$(0,i_1)$ and this minimum is negative. Then $i_1=1$. Further,
$b(u)$ has a local maximum at $u_0\in (i_2,i_3)$ and this maximum is greater than 1.
For $u \geqslant u_0$, $b(u)$ is strictly decreasing. It yields that
$i_3=n-d+1$. This choice of $i_1, i_2, i_3$  only implies
that $b_u \geqslant 0$ for  $u \leqslant r-d$ and
$b_u \geqslant 1$ for $u \geqslant r-d+1$. 

For $p_r$, we have $b_{i_1}=0$,
$b_{r-d+1}=1$ and $b_{i_3}=0$ for some $1\leqslant i_1 \leqslant r-d$ and
$r-d+1  < i_3 \leqslant n-d+1$.
Then $b(u)$ has local extremes on $(0,i_1)$ and $(i_1, i_3)$.
Moreover, $b(u)$ has a local minimum on 
$(0,i_1)$ and this minimum is negative. Therefore $i_1=1$. Further,
$b(u)$ has a local maximum on $(i_1,i_3)$ and this maximum is greater or equal to 1.
It yields that $i_3=n-d+1$. This choice of $i_1$ and $ i_3$  only yields
that $b_u \geqslant 0$ for  $u \neq r-d+1$ and
$b_{r-d+1} \geqslant 1$. 

It follows that $\mathbf{i}$ is the same, but $\mathbf{v}_{\mathbf{i}}$ in  (\ref{20})
is different for $P_r$ and $p_r$.

Put $\mathbf{i}=(1,r-d+1,n-d+1)$. Then
$$ \mathbf{F}_\mathbf{i} = \left\lgroup \begin{matrix}
1 & C_{r}^d & C_n^d\cr  0 & C_{r}^{d+1} & C_n^{d+1} \cr  0 & C_{r}^{d+2} & C_n^{d+2}
\end{matrix} 
\right\rgroup, \quad
\mathbf{v}_\mathbf{i} = \left\lgroup \begin{matrix}
0 \cr 1 \cr 1
\end{matrix} \right\rgroup
\quad \mbox{or}\quad
\mathbf{v}_\mathbf{i} = \left\lgroup \begin{matrix}
0 \cr 1 \cr 0
\end{matrix} \right\rgroup
$$
for $P_r$ and $p_r$, correspondingly.

The solutions of system (\ref{20}) are $  \mathbf{a} =\bm{\beta} $
for $P_r$ and $  \mathbf{a} =\bm{\delta} $ for $p_r$.
Then we obtain (\ref{Pr3ub2}).

d) In case of $P_r$, we have  $b_{i_1}=b_{i_2}=b_{i_3}=1$ for some 
$r-d+1 \leqslant i_1< i_2 < i_3 \leqslant n-d+1$. 
Then $b(u)$ has local extremes on $(i_1,i_2)$ and $(i_2, i_3)$.
Hence, $b(u)$ is strictly increasing on $(0, i_1)$. 
It follows that $i_1=r-d+1$. 
Moreover, $b(u)$ has a local maximum on  $(i_1,i_2)$ and 
a local minimum on $(i_2,i_3)$. This local minimum is less than 1.
Hence $i_2=i_1+1$.
This choice of $i_1, i_2, i_3$ only yields
that $b_u \geqslant 0$ for  $u \leqslant r-d$ and
$b_u \geqslant 1$ for $u \geqslant r-d+1$. 

For $p_r$, we have 
$b_{r-d+1}=1$ and $b_{i_2}=b_{i_3}=0$ for some 
$r-d+1 < i_2 < i_3 \leqslant n-d+1$.
Then $b(u)$ has local extremes on $(0,i_2)$ and $(i_2, i_3)$.
Moreover, $b(u)$ has a local minimum on 
$(i_2,i_3)$ and this minimum is negative. Therefore $i_3=i_2+1$. Further,
$b(u)$ has a local maximum on $(0,i_2)$ and this maximum is greater or equal to 1.
This choice of $i_1$ and $ i_3$ only gives $b_u \geqslant 0$ for  $u \neq r-d+1$ and
$b_{r-d+1} \geqslant 1$. 

It follows that $\mathbf{i}$ is the same, but $\mathbf{v}_{\mathbf{i}}$ in  (\ref{20}) is different for
$P_r$ and $p_r$.

Put $\mathbf{i}=(r-d+1,m,m+1)$, where $r-d+2 \leqslant m \leqslant n-d$.
Then
$$ \mathbf{F}_\mathbf{i} = \left\lgroup \begin{matrix}
 C_r^{d} & C_{m+d-1}^{d} & C_{m+d}^{d}  \cr
 C_r^{d+1} & C_{m+d-1}^{d+1} & C_{m+d}^{d+1} \cr
 C_r^{d+2} & C_{m+d-1}^{d+2} & C_{m+d}^{d+2}  
\end{matrix} \right\rgroup, \quad 
\mathbf{v}_\mathbf{i} = \left\lgroup \begin{matrix}
1 \cr 1 \cr 1
\end{matrix} \right\rgroup
 \quad\mbox{or}\quad
\mathbf{v}_\mathbf{i} = \left\lgroup \begin{matrix}
1 \cr 0 \cr 0
\end{matrix} \right\rgroup
$$
for $P_r$ and $p_r$, correspondingly.

Note that 
$$ \mathbf{F}^{-1}_\mathbf{i} = \left\lgroup \begin{matrix}
\frac{m(m-1)}{C_{r}^d \Delta} & 
-\frac{2(m-1)(d+1)}{C_{r}^d \Delta } &
\frac{(d+1)(d+2)}{C_{r}^d \Delta }  \cr
 \frac{m(r-d)(r-d-m)}{C_{m+d-1}^d \Delta }& 
-\frac{(r-d-m)(r-d+m-1)(d+1)}{C_{m+d-1}^d \Delta }& 
\frac{(r-d-m)(d+1)(d+2)}{C_{m+d-1}^d \Delta } \cr
 - \frac{(m-1)(r-d)(r-d-m+1)}{C_{m+d}^d \Delta }& 
 \frac{(r-d-m+1)(r-d+m-2)(d+1)}{C_{m+d}^d \Delta }& 
- \frac{(r-d-m+1)(d+1)(d+2)}{C_{m+d}^d \Delta } 
\end{matrix} 
\right\rgroup. 
$$

The solutions of system (\ref{20}) are $ \mathbf{a}=\bm{\gamma} $ 
and $ \mathbf{a}=\bm{\alpha} $ and we get (\ref{Pr3ub3}).

One can check that $\gamma_1=1$ and $\gamma_2=\gamma_3=0$ for $d=0$ and
$\gamma_1>0$, $\gamma_2<0$ and $\gamma_3>0$ for $d>0$. 

Make an optimization over $m$.  By (\ref{30}), we have
$\mathbf{z}^*_{\mathbf{i}}(j) = \mathbf{F}^{-1}_\mathbf{i} \mathbf{s}(j)$.
Hence,
\begin{eqnarray*}
&& \hspace*{-\parindent}
z^*_{r-d+1}(j) = \frac{1}{C_{r}^d \Delta} \left(
m(m-1) s_1(j) -2(m-1)(d+1) s_2(j) + (d+1)(d+2) s_3(j) \right),  
\\ && \hspace*{-\parindent}
z^*_{m}(j) = \frac{r-d-m}{C_{m+d-1}^d \Delta } \left(
m(r-d) s_1(j) -(r-d+m-1)(d+1) s_2(j) + (d+1)(d+2) s_3(j) \right),
\\ && \hspace*{-\parindent}
z^*_{m+1}(j) = \frac{r-d-m+1}{C_{m+d}^d \Delta } \left(
 - (m-1)(r-d) s_1(j)+  
\right. \\ && \hspace*{9\parindent} \left. + 
 (r-d+m-2)(d+1) s_2(j) - (d+1)(d+2) s_3(j) \right).
\end{eqnarray*}
The inequalities $ z^*_{m}(j)\geqslant 0 $ and $ z^*_{m+1}(j)\geqslant 0 $ yield (\ref{mop})
provided 
$ (d+1) s_2(j) - (r-d) s_1(j)>0 $.


Inequalities (\ref{Pr3ub1})--(\ref{Pr3ub3}) were derived by an estimation of items of the decomposition
of $P_r$ and $p_r$ for every $j$. It is clear that every such items can be underestimated
by a minimum of three bounds considered above. So, inequality (\ref{Pr3ub}) follows.

For conditional probabilities, the argument is the same as that in Theorem \ref{th2}.
Hence, we omit details.
\end{proof}

Now, we turn to lower bounds for probabilities of combinations of events.
To this end, we have the next result.

\begin{theorem}\label{th4}
Assume that $\ell=3$ and $\mathbf{s}(j) = \mathbf{F} \mathbf{z}(j)$ for all $j\in J_d$
with $\mathbf{F}$ from (\ref{f}). (In this case, relation (\ref{sk}) holds for
$s_1(j)$, $s_2(j)$ and $s_3(j)$, $j\in J_d$.) 

If $r-d \geqslant 2$, 
then
\begin{eqnarray}
\label{Pr3lb1}
P_r \geqslant  \sum_{j \in J_d} {\bm\alpha}^{T} \mathbf{s}(j)\quad\mbox{and}\quad
p_n \geqslant  \sum_{j \in J_d} {\bm\delta}^{T} \mathbf{s}(j),
\end{eqnarray}
where
$$ \bm\alpha = \frac{1}{\Delta_2}
\left\lgroup \begin{matrix}
0 \cr -C_{r-1}^{d+2} \cr C_{r-1}^{d+1} 
\end{matrix} \right\rgroup,\quad \Delta_2= C_n^{d+2} C_{r-1}^{d+1} - C_n^{d+1} C_{r-1}^{d+2}
$$
and $\bm \delta = \bm \alpha$ 
with $r=n$.

If $r-d \geqslant 1$ and $n-r \geqslant 1$, then
\begin{eqnarray}
\label{Pr3lb2}
P_r \geqslant  \sum_{j \in J_d} {\bm\beta}^{T} \mathbf{s}(j),
\quad\mbox{and}\quad
p_r  \geqslant  \sum_{j \in J_d} {\bm\theta}^{T} \mathbf{s}(j),
\end{eqnarray}
where 
\begin{eqnarray*}
\bm\theta = \frac{1}{C_{r}^d  }
\left\lgroup \begin{matrix}
 - (r-d+1)(r-d-1)\cr
(d+1) (2(r-d)-1)\cr
-(d+1)(d+2)
\end{matrix} \right\rgroup
\end{eqnarray*}
and $ {\bm\beta} =(\beta_1,\beta_2, \beta_3)^T $ with
\begin{eqnarray*}
&& \beta_1= \frac{m(r-d-1)(r-d-m-1)}{C_{m+d-1}^d \Delta_3 }-\frac{(m-1)(r-d-1)(r-d-m)}{C_{m+d}^d \Delta_3 }
,\\
&& \beta_2= (d+1)\left(
-\frac{(r-d-m-1)(r-d+m-2)}{C_{m+d-1}^d \Delta_3 } + \frac{(r-d-m)(r-d+m-3)}{C_{m+d}^d \Delta_3 }
\right),\\
&& \beta_3= (d+1)(d+2)\left(\frac{r-d-m-1}{C_{m+d-1}^d \Delta_3 } - \frac{r-d-m}{C_{m+d}^d \Delta_3 }
\right),
\end{eqnarray*}
 $\Delta_3 = (r-d-m-1)(r-d-m)$
and $m$ is an arbitrary natural number such that $r-d+1 \leqslant m \leqslant n-d$. 
Optimal value of $m$ is
\begin{eqnarray}
\label{mop1}
m-1  \leqslant (d+1)
\frac{ (d+2) s_3(j)-(r-d-2) s_2(j)}{(d+1) s_2(j) -(r-d-1) s_1(j) } \leqslant m
\end{eqnarray}
provided $ (d+1) s_2(j) - (r-d-1) s_1(j)>0 $
and $r-d+1 $ or $n-d$ otherwise.

If 
$n-r \geqslant 2$ and $r=d$, then
\begin{eqnarray}
\label{Pr3lb3}
P_d \geqslant  \sum_{j \in J_d} {\bm\gamma}^{T} \mathbf{s}(j)
\quad\mbox{and}\quad
p_d \geqslant  \sum_{j \in J_d} {\bm\varphi}^{T} \mathbf{s}(j),
\end{eqnarray}
where  
\begin{eqnarray*}
\bm\varphi = 
\left\lgroup \begin{matrix}
 1 \cr
-(d+1) \cr
\frac{(d+1)(d+2)}{n-d}
\end{matrix} \right\rgroup,
\end{eqnarray*}
$ {\bm\gamma} = (\gamma_{1}, \gamma_{2}, \gamma_{3})^T $ with
\begin{eqnarray*}
&& 
\hspace*{-\parindent}
\gamma_1= \frac{m(n-d)(n-d-m)}{C_{m+d-1}^d \Delta_4 } - \frac{(m-1)(n-d)(n-d-m+1)}{C_{m+d}^d \Delta_4 }
+\frac{m(m-1)}{C_{n}^d \Delta_4 },\\
&& 
\hspace*{-\parindent}
\gamma_2= (d+1)\left(
-\frac{(n-d-m)(n-d+m-1)}{C_{m+d-1}^d \Delta_4 } + \frac{(n-d-m+1)(n-d+m-2)}{C_{m+d}^d \Delta_4 }
-\frac{2(m-1)}{C_{n}^d \Delta_4 }\right),\\
&& 
\hspace*{-\parindent}
\gamma_3= (d+1)(d+2)\left(
\frac{n-d-m}{C_{m+d-1}^d \Delta_4 } - \frac{n-d-m+1}{C_{m+d}^d \Delta_4 }
+\frac{1}{C_{n}^d \Delta_4 }\right),
\end{eqnarray*}
$\Delta_4 = (n-d-m)(n-d-m+1)$  and $m$ is an arbitrary natural number such that 
$1 \leqslant m \leqslant n-d-1$.
Optimal value of $m$ is
\begin{eqnarray}
\label{mop2}
 m-1  \leqslant (d+1)
\frac{(n-d-1) s_2(j)- (d+2) s_3(j)}{ (n-d) s_1(j) -(d+1) s_2(j)} \leqslant m
\end{eqnarray}
provided $(n-d) s_1(j)-(d+1) s_2(j)>0 $
and $1$ or $n-d-1$ otherwise.

If  $r-d \geqslant 2$ and $n-r \geqslant 1$, then
\begin{eqnarray}
\label{Pr3lb}
P_r \geqslant 
\sum_{j \in J_d} \max\left\{ {\bm\alpha}^{T} \mathbf{s}(j),
{\bm\beta}^{T} \mathbf{s}(j), {\bm\gamma}^{T} \mathbf{s}(j)
\right\}.
\end{eqnarray}

Define $s_1^\mathcal{A}(j)$, $s_2^\mathcal{A}(j)$ and $s_3^\mathcal{A}(j)$ for $j\in J_d$  by 
the right-hand side of relation (\ref{sk}) with $\mathbf{P}(\cdot)$ replaced
by $\mathbf{P}(\cdot|\mathcal{A})$. Inequalities (\ref{Pr3lb1}), (\ref{Pr3lb2}),
(\ref{Pr3lb3}) and (\ref{Pr3lb}) hold a.s.
provided one replaced $P_r$ and $p_r$ by  $P_r^{\mathcal{A}}$ and $p_r^{\mathcal{A}}$ and
$s_1(j)$, $s_2(j)$ and $s_3(j)$ by
$s_1^\mathcal{A}(j)$, $s_2^\mathcal{A}(j)$ and $s_3^\mathcal{A}(j)$
correspondingly. In this case, optimal $m$ are random variables.
\end{theorem}

\begin{proof}
We will use properties of polynomial $b(u)$ from (\ref{bu}) to find $\mathbf{i}$.
As in the proof of Theorem \ref{th3}, we deal with two variants of vector $\mathbf{v}$
from (\ref{vv})
to get bounds for $P_r$ and $p_r$, correspondingly. 
Vector $\mathbf{v}_\mathbf{i}$ can again be as in (\ref{ad})
for $P_r$ and $p_r$, correspondingly. We deal 
with every case from a)--d) of (\ref{ad})
for $P_r$ and $p_r$ simultaneously.

If we derive lower bounds for $P_r$, then we need  $b_u\leqslant 0$ for $u \leqslant r-d$ and $b_u\leqslant 1$ for $u \geqslant r-d+1$. While we obtain such bounds for $p_r$, we need 
$b_u \leqslant 0$ for  $u \neq r-d+1$ and $b_{r-d+1} \leqslant 1$.

a) In this case $a_1=a_2=a_3=0$ and, therefore, $b_u=0$ for all $u$.
So, we have a trivial lower bounds by zero for $P_r$ and $p_r$ both.

b) In case of $P_r$, we have 
$b_{i_1}=b_{i_2}=0$ and $b_{i_3}=1$ for some $i_1<i_2 \leqslant r-d < i_3 \leqslant n-d+1$. 
Then there are two extremes of $b(u)$ on intervals $(0,i_1)$ and $(i_1,i_2)$. Hence, $b(u)$ is
strictly increasing on $(i_2, i_3)$. This yields that $i_3=n-d+1$. Moreover,
$b(u)$ has a local minimum on $(i_1,i_2)$ and this minimum is negative. It follows
that $i_2=r-d$. Further, $b(u)$ has a local maximum on $(0,i_1)$ and
this maximum is positive. Hence, $i_1=1$. This choice of $i_1, i_2, i_3$ only  implies
that $b_u \leqslant 0$ for  $u \leqslant r-d$ and
$b_u \leqslant 1$ for $u \geqslant r-d+1$. 

For $p_r$, we have
$b_{i_1}=b_{i_2}=0$ and $b_{r-d+1}=1$ for some $i_1<i_2 \leqslant r-d $.
Then there are two extremes of $b(u)$ on intervals $(0,i_1)$ and $(i_1,i_2)$. Hence, $b(u)$ is
strictly increasing on $(i_2, r-d+1)$. This yields that 
$r=n$ and $i_2=r-d$. Moreover, $b(u)$ has a local minimum on $(i_1,i_2)$ and this minimum is negative. 
Further, $b(u)$ has a local maximum on $(0,i_1)$ and
this maximum is positive. Hence, $i_1=1$. This choice of $i_1, i_2$  only yields
that $b_u \leqslant 0$ for  $u \neq r-d+1$ and
$b_{r-d+1} \leqslant 1$.

It follows that $\mathbf{i}$ and $\mathbf{v}_{\mathbf{i}}$ in (\ref{20}) are the same, but
we have $r=n$ for $p_r$.

Put $\mathbf{i}=(1,r-d,n-d+1)$. Then
$$ \mathbf{F}_\mathbf{i} = \left\lgroup \begin{matrix}
1 & C_{r-1}^d & C_n^d\cr  0 & C_{r-1}^{d+1} & C_n^{d+1} \cr  0 & C_{r-1}^{d+2} & C_n^{d+2}
\end{matrix} 
\right\rgroup \quad\mbox{and}\quad
\mathbf{v}_\mathbf{i} = \left\lgroup \begin{matrix}
0 \cr 0 \cr 1
\end{matrix} \right\rgroup.
$$

The solution of system (\ref{20}) is $ \mathbf{a}=\bm\alpha $ and we get (\ref{Pr3lb1}).

c) For $P_r$, we have $b_{i_1}=0$ and
$b_{i_2}=b_{i_3}=1$ for some $1\leqslant i_1 \leqslant r-d < i_2 < i_3 \leqslant n-d+1$.
Then $b(u)$ has local extremes on $(0,i_1)$ and $(i_2, i_3)$.
Hence, $b(u)$ is strictly increasing on $(i_1, i_2)$. Moreover, $b(u)$ has a local minimum on 
$(0,i_1)$ and this minimum is negative. Then $i_1=r-d$. Further,
$b(u)$ has a local maximum at $u_0\in (i_2,i_3)$ and this maximum is greater than 1.
It yields that $i_3=i_2+1$. This choice of $i_1, i_2, i_3$ only gives
$b_u \leqslant 0$ for  $u \leqslant r-d$ and
$b_u \leqslant 1$ for $u \geqslant r-d+1$. 

For $p_r$, we have $b_{i_1}=0$,
$b_{r-d+1}=1$ and $b_{i_3}=0$ for some $1\leqslant i_1 \leqslant r-d$ and
$r-d+1  < i_3 \leqslant n-d+1$.
Then $b(u)$ has local extremes on $(0,i_1)$ and $(i_1, i_3)$.
Moreover, $b(u)$ has a local minimum on 
$(0,i_1)$ and this minimum is negative. Therefore $i_1=r-d$. Further,
$b(u)$ has a local maximum on $(i_1,i_3)$ and this maximum is greater or equal to 1.
It yields that $i_3=r-d+2$. This choice of $i_1$ and $ i_3$   implies
that $b_u \leqslant 0$ for  $u \neq r-d+1$ and
$b_{r-d+1} \leqslant 1$. 
The matrix $\mathbf{F}_{\mathbf{i}}$ is
the same as that for $P_r$ when $i_2=r-d+1$. Vector $\mathbf{v}$ is different of course.

For $P_r$, put $\mathbf{i}=(r-d,m,m+1)$, where $r-d+1 \leqslant m \leqslant n-d$.
Then
$$ \mathbf{F}_\mathbf{i} = \left\lgroup \begin{matrix}
 C_{r-1}^{d} & C_{m+d-1}^{d} & C_{m+d}^{d}  \cr
 C_{r-1}^{d+1} & C_{m+d-1}^{d+1} & C_{m+d}^{d+1} \cr
 C_{r-1}^{d+2} & C_{m+d-1}^{d+2} & C_{m+d}^{d+2}  
\end{matrix} 
\right\rgroup \quad\mbox{and}\quad
\mathbf{v}_\mathbf{i} = \left\lgroup \begin{matrix}
0 \cr 1 \cr 1
\end{matrix} \right\rgroup.
$$
For $p_r$, we put $m=r-d+1$ and $\mathbf{v} = (0,1,0)^T$.

Note that
$$ \mathbf{F}^{-1}_\mathbf{i} = \left\lgroup \begin{matrix}
\frac{m(m-1)}{C_{r-1}^d \Delta_3} & 
-\frac{2(m-1)(d+1)}{C_{r-1}^d \Delta_3 } &
\frac{(d+1)(d+2)}{C_{r-1}^d \Delta_3 }  \cr
 \frac{m(r-d-1)(r-d-m-1)}{C_{m+d-1}^d \Delta_3 }& 
-\frac{(r-d-m-1)(r-d+m-2)(d+1)}{C_{m+d-1}^d \Delta_3 }& 
\frac{(r-d-m-1)(d+1)(d+2)}{C_{m+d-1}^d \Delta_3 } \cr
 - \frac{(m-1)(r-d-1)(r-d-m)}{C_{m+d}^d \Delta_3 }& 
 \frac{(r-d-m)(r-d+m-3)(d+1)}{C_{m+d}^d \Delta_3 }& 
- \frac{(r-d-m)(d+1)(d+2)}{C_{m+d}^d \Delta_3 } 
\end{matrix} 
\right\rgroup, 
$$

The solutions of system (\ref{20}) are $\mathbf{a}=\bm\beta$ for $P_r$ and 
 $\mathbf{a}=\bm\theta$ for $p_r$ and we obtain (\ref{Pr3lb2}).

One can check that $\beta_1<0$, $\beta_2>0$ and $\beta_3<0$.
 
Turn to an optimization over $m$.  By (\ref{30}), we have
$\mathbf{z}^*_{\mathbf{i}}(j) = \mathbf{F}^{-1}_\mathbf{i} \mathbf{s}(j)$.
Hence, we obtain
we have
\begin{eqnarray*}
&& \hspace*{-\parindent}
z^*_{r-d}(j) = \frac{1}{C_{r-1}^d \Delta_3} \left(
m(m-1) s_1(j) -2(m-1)(d+1) s_2(j) + (d+1)(d+2) s_3(j) \right),  
\\ &&\hspace*{-\parindent}
z^*_{m}(j) = \frac{r-d-m-1}{C_{m+d-1}^d \Delta_3 } \left(
m(r-d-1) s_1(j) -(r-d+m-2)(d+1) s_2(j) + 
\right. \\ && \hspace*{18\parindent} \left. +
(d+1)(d+2) s_3(j) \right),
\\ &&\hspace*{-\parindent}
z^*_{m+1}(j) = \frac{r-d-m}{C_{m+d}^d \Delta_3 } \left(
 - (m-1)(r-d-1) s_1(j)+  
\right. \\ && \hspace*{9\parindent} \left. +
 (r-d+m-3)(d+1) s_2(j) - (d+1)(d+2) s_3(j) \right).
\end{eqnarray*}
The inequalities $ z^*_{m}(j)\geqslant 0 $ and $ z^*_{m+1}(j)\geqslant 0 $ implies (\ref{mop1})
provided 
$ (d+1) s_2(j) >(r-d-1) s_1(j) $.


d) For $P_r$, we have  $b_{i_1}=b_{i_2}=b_{i_3}=1$ for some 
$r-d+1 \leqslant i_1< i_2 < i_3 \leqslant n-d+1$. 
Then $b(u)$ has local extremes on $(i_1,i_2)$ and $(i_2, i_3)$.
Hence, $b(u)$ is strictly increasing on $(0, i_1)$. This yields that $r-d=0$.
Moreover, $b(u)$ has a local maximum on  $(i_1,i_2)$ and 
a local minimum on $(i_2,i_3)$. This local minimum is less than 1.
Hence, $i_3=n-d+1$ and $i_2=i_1+1$.
This choice of $i_1, i_2, i_3$ only  implies
that $b_u \leqslant 0$ for  $u \leqslant r-d$ and
$b_u \leqslant 1$ for $u \geqslant r-d+1$. 

For $p_r$, we have 
$b_{r-d+1}=1$ and $b_{i_2}=b_{i_3}=0$ for some 
$r-d+1 < i_2 < i_3 \leqslant n-d+1$.
Then $b(u)$ has local extremes on $(0,i_2)$ and $(i_2, i_3)$.
Moreover, $b(u)$ has a local minimum on 
$(i_2,i_3)$ and this minimum is negative. Further,
$b(u)$ has a local maximum on $(0,i_2)$ and this maximum is greater or equal to 1.
It follows that $b(u)$ is positive in $(0, i_2)$.  Hence, $r-d=0$ and $i_2=r-d+2=2$. 
Furthermore, $b(u)$ is increasing for $u \geqslant i_3$. It yields that $i_3=n-d+1$.
This choice of $i_1$ and $ i_3$   implies
that $b_u \leqslant 0$ for  $u \neq r-d+1$ and
$b_{r-d+1} \leqslant 1$. 
The matrix $\mathbf{F}_\mathbf{i}$ is the same as that for $P_r$ with
$i_1=1$ and $r=d$. 

Take $r=d$.
For $P_r$, put $\mathbf{i}=(m,m+1, n-d+1)$, where $r-d+1 \leqslant m \leqslant n-d-1$.
Then
$$ \mathbf{F}_\mathbf{i} = \left\lgroup \begin{matrix}
C_{m+d-1}^{d} & C_{m+d}^{d} &  C_n^{d}  \cr
 C_{m+d-1}^{d+1} & C_{m+d}^{d+1} & C_n^{d+1} \cr
C_{m+d-1}^{d+2} & C_{m+d}^{d+2}  & C_n^{d+2} 
\end{matrix} 
\right\rgroup \quad\mbox{and}\quad
\mathbf{v}_\mathbf{i} = \left\lgroup \begin{matrix}
1 \cr 1 \cr 1
\end{matrix} \right\rgroup.
$$
For $p_r$, we take $m=1$ and $\mathbf{v}=(1,0,0)^T$.

Note that
$$ \mathbf{F}^{-1}_\mathbf{i} = \left\lgroup \begin{matrix}
 \frac{m(n-d)(n-d-m)}{C_{m+d-1}^d \Delta_4 }& 
-\frac{(n-d-m)(n-d+m-1)(d+1)}{C_{m+d-1}^d \Delta_4 }& 
\frac{(n-d-m)(d+1)(d+2)}{C_{m+d-1}^d \Delta_4 } \cr
 - \frac{(m-1)(n-d)(r-d-m+1)}{C_{m+d}^d \Delta_4 }& 
 \frac{(n-d-m+1)(n-d+m-2)(d+1)}{C_{m+d}^d \Delta_4 }& 
- \frac{(n-d-m+1)(d+1)(d+2)}{C_{m+d}^d \Delta_4 } \cr
\frac{m(m-1)}{C_{n}^d \Delta_4} & 
-\frac{2(m-1)(d+1)}{C_{n}^d \Delta_4 } &
\frac{(d+1)(d+2)}{C_{n}^d \Delta_4 }  
\end{matrix} 
\right\rgroup. 
$$

The solutions of system (\ref{20}) are $ \mathbf{a}=\bm\gamma $ for $P_r$
$ \mathbf{a}=\bm\varphi $ for $p_r$
and we arrive at (\ref{Pr3lb3}).

It is not difficult to check that $\gamma_1=1$ and $\gamma_2=a_3=0$ for $d=0$. 
For $d>0$, one can check that $\gamma_1>0$, $\gamma_2<0$ and $\gamma_3>0$. 

Optimize over $m$.  By (\ref{30}), we have
$\mathbf{z}^*_{\mathbf{i}}(j) = \mathbf{F}^{-1}_\mathbf{i} \mathbf{s}(j)$.
Hence, we get
\begin{eqnarray*}
&& \hspace*{-\parindent}
z^*_{n-d+1}(j) = \frac{1}{C_{n}^d \Delta_4} \left(
m(m-1) s_1(j) -2(m-1)(d+1) s_2(j) + (d+1)(d+2) s_3(j) \right),  
\\ && \hspace*{-\parindent}
z^*_{m}(j) = \frac{n-d-m}{C_{m+d-1}^d \Delta_4 } \left(
m(n-d) s_1(j) -(n-d+m-1)(d+1) s_2(j) + (d+1)(d+2) s_3(j) \right),
\\ && \hspace*{-\parindent}
z^*_{m+1}(j) = \frac{n-d-m+1}{C_{m+d}^d \Delta_4 } \left(
 - (m-1)(n-d) s_1(j)+  
 \right. \\ && \hspace*{9\parindent} \left. +
 (n-d+m-2)(d+1) s_2(j) - (d+1)(d+2) s_3(j) \right).
\end{eqnarray*}
The inequalities $ z^*_{m}(j)\geqslant 0 $ and $ z^*_{m+1}(j)\geqslant 0 $ yield (\ref{mop2})
provided 
$(n-d) s_1(j)-(d+1) s_2(j)>0 $.


Inequality (\ref{Pr3lb}) follows by the proof above.

For conditional probabilities, the argument is the same as that in Theorem \ref{th2}.
Hence, we omit details.
\end{proof}

For $d=0$, one can find inequalities of Theorems \ref{th2}--\ref{th4} in [16].
For $d>0$, the results are new. 

Similar inequalities can be obtained for measurable space
$(\Omega, \mathcal{F}, \mu)$ with $\mu(\Omega)<\infty$ by introducing of probability 
$\mathbf{P}(\cdot)=  \mu(\cdot)/\mu(\Omega)$. For $\sigma$-finite $\mu$,
we have $\Omega = \bigcup_{k=1}^\infty \Omega_k$ with $\Omega_k\cap \Omega_i = \emptyset$
for $k\neq i$ and $\mu(\Omega_k)<\infty$ for all $k$. For $B\in \mathcal{F}$, we can get
bounds for $\mu(B\cap \Omega_k)$ for every $k$ and derive a bound for $B$ then.
In particular, it works for the Lebesgue measure.
If $\mu$ is a measure with sign, we have
$\mu=\mu^+-\mu^-$, where $\mu^+$ and $\mu^-$ are two finite measures concentrated on
$\Omega^+$ and $ \Omega^-$ with $\Omega^+\cap \Omega^- = \emptyset$ and
$\Omega = \Omega^+\cup\Omega^-$. Hence, $\mu(B)=\mu^+(B)-\mu^-(B)$ and
$\mu^+(B)$ and $\mu^-(B)$ can be underestimated separately as mentioned 
before. Moreover, our approach yields that bounds will be sharp, i.e.
they can turn to equalities for some sets.



\bigskip
\noindent
{\bf References}
{\parindent0mm
{\begin{itemize}
\footnotesize
\item[{[1]}]
Frolov A.N., 2017. On inequalities for probabilities wherein at least $r$ from $n$ events occur.
Vestnik Sankt-Peterburgskogo Universiteta: Matematika, Mekhanika, Astronomiya, 2017, Vol. 62, No. 3, pp. 
477--478. (In Russian). 
English translation: Vestn. St. Petersburg Univ.: Math. 50, 287--296 (2017),
Allerton Press, Inc.
\item[{[2]}]
Frolov A.N., 2018. On inequalities for probabilities of joint occurrence of several events. 
Vestnik Sankt-Peterburgskogo Universiteta: Matematika, Mekhanika, Astronomiya, 2018, Vol. 63, No. 3, pp. 
464--476. (In Russian). 
English translation: Vestn. St. Petersburg Univ.: Math. 51, 286--295 (2018).
Allerton Press, Inc.
\item[{[3]}]
Frolov A.N., 2012.
Bounds for probabilities of unions of events and the Borel--Cantelli lemma.
Statist. Probab. Lett. 82, 2189--2197.
\item[{[4]}]
Frolov A.N., 2015. On lower and upper bounds for probabilities of unions and the Borel---Cantelli lemma.
Studia Sci. Math. Hungarica. 52 (1), 102--128.
\item[{[5]}]
Frolov A.N., 2019. On bounds for probabilities of combinations of events, Jordan's formula  
and Bonferroni inequalities. 
Vestnik Sankt-Peterburgskogo Universiteta: Matematika, Mekhanika, Astronomiya, 2019, Vol. 64, No. 2, pp. 253--265 (In Russian). 
English translation: 
Vestnik St. Petersburg University, Mathematics, 2019, Vol. 52, No. 2, pp. 178--186. 
Allerton Press, Inc.
\item[{[6]}]
Chung K.L., Erd\H{o}s P., 1952. On the application of the Borel--Cantelli lemma.
Trans. Amer. Math. Soc. 72, 179--186.
\item[{[7]}]
Gallot S. 1966. A bound for the maximum of a number of random variables.
J.~Appl.~Probab. 3, 556--558.
\item[{[8]}]
Dawson D.A., Sankoff D., 1967. An inequality for probabilities.
Proc.~Amer.~Math.~Soc. 18, 504--507.
\item[{[9]}]
Kounias E.G., 1968. Bounds for the probability of a union, 
with applications.
Ann. Math. Statist. 39, 2154--2158.
\item[{[10]}]
Kwerel S.M., 1975a. Bounds on the probability of the union and 
intersection of $m$ events.
Adv. Appl. Probab. 7, 431--448.
\item[{[11]}]
Kwerel S.M., 1975b. Most stringent bounds on aggregated probabilities of partially specified dependent probability systems. J. of Amer. Statist. Assoc., 70, 472--479.
\item[{[12]}]
Kwerel S.M., 1975c. Most stringent bounds on the probability of the union and intersection of m events for systems partially specified by $S_{1}, S_{2}, \cdots S_{k}, 2 \leqq k < m$.
J. of Appl. Probab., 12, 612--619. 
\item[{[13]}]
M\'{o}ri T.F., Sz\'{e}kely  G.J., 1985.
A note on the background of several Bonferroni--Galambos-type inequalities. J. of Appl. Probab., 22, 836--843.
\item[{[14]}]
Boros E., Pr\'{e}kopa A., 1989.  Closed form two-sided bounds for probabilities
that at least $r$ and exactly $r$ out of $n$ events occurs.
Math. Oper. Research. 14, 317--342.
\item[{[15]}]
Sibuya M., 1991.  Bonferroni-type inequalities; Chebyshev-type inequalities
for distributions on $[0,n]$. Ann. Inst. Statist. Math., 43,  2, 261--285. 
\item[{[16]}]
Kounias S., Sotirakoglou K., 1993. Upper and lower bounds for
the probability that $r$ events occur. J. Math. Programming.
Oper. Research., 27, 1-2, 63--78.
\item[{[17]}]
Galambos J., Simonelli I., 1996. Bonferroni-type inequalities with applications.
Springer-Verlag N.Y.
\item[{[18]}]
de Caen D., 1997. A lower bound on the probability of a union.
Discrete Math. 169, 217--220.
\item[{[19]}]
Kuai H., Alajaji F., Takahara G., 2000. A lower bound on the probability of a finite union of events.
Discrete Math. 215, 147--158.
\item[{[20]}]
Pr\'{e}kopa A.,Gao L. 2005. Bounding the probability of the union of events
by aggregation and disaggregation in linear programs.
Discrete Appl. Math. 145, 444--454.
\item[{[21]}]
Frolov A.N., 2012.
Bounds for probabilities of unions of events and the Borel--Cantelli lemma.
Statist. Probab. Lett. 82, 2189--2197.
\item[{[22]}]
Frolov A.N., 2014. On inequalities for probabilities of unions 
of events and the Borel--Cantelli lemma.
Vestnik Sankt-Peterburgskogo Universiteta, Seriya 1. Matematika, 
Mekhanika, Astronomiya, N 2, 201--210. (In Russian) 
English translation: Vestnik St.Petersburg University, 
Mathematics, 2014, 47, N 2, 68--75.
Allerton Press, Inc.
\item[{[23]}]
Frolov A.N., 2015. On lower and upper bounds for probabilities of unions and the Borel---Cantelli lemma.
Studia Sci. Math. Hungarica. 52 (1), 102--128.
\item[{[24]}]
Frolov A.N., 2015b. On estimation of probabilities of
unions of events with applications to the Borel--Cantelli 
lemma. Vestnik Sankt-Peterburgskogo Universiteta, Seriya 1. 
Matematika, Mekhanika, Astronomiya, N 3, 399--404.
English translation: Vestnik St.Petersburg University, Mathematics, 2015, 48, N 3, 175--180.
Allerton Press, Inc.
\item[{[25]}]
A. N. Frolov, On inequalities for conditional probabilities of unions of events and the conditional Borel--Cantelli lemma. 
Vestnik Sankt-Peterburgskogo Universiteta: Matematika, Mekhanika, Astronomiya, 2016, Vol. 61, No. 4, pp. 
651--662. (In Russian).
English translation: 
Vestn. St. Petersburg Univ.: Math. 49, 379--388 (2016).
Allerton Press, Inc.
\item[{[26]}]
Frolov A.N., 2017. On inequalities for values of first jumps of distribution functions and H\"{o}lder's 
inequality. Statist. Probab. Lett. 126, 150-156.

\end{itemize}
}
}

\end{document}